\documentclass[11pt]{amsart}
\usepackage{amsfonts, amssymb, amsmath, amsthm, color, float,enumerate}
\usepackage[unicode,psdextra]{hyperref}
\usepackage{bookmark}
\usepackage{url}
\usepackage{pgfplots,tikz}
\usetikzlibrary{arrows}

\theoremstyle{plain}
   \newtheorem{teo}{Theorem}
   \newtheorem{coro}[teo]{Corollary}
   \newtheorem{lema}[teo]{Lemma}

\theoremstyle{definition}
   
\theoremstyle{remark}
 \newtheorem{obs}{Remark}

\numberwithin{equation}{section}
\numberwithin{teo}{section}
\allowdisplaybreaks

\definecolor{zzttqq}{rgb}{0.6,0.2,0.}
\definecolor{qqzzqq}{rgb}{0.,0.6,0.}
\definecolor{aquamarine}{rgb}{0.5, 1.0, 0.83}
\definecolor{blizzardblue}{rgb}{0.67, 0.9, 0.93}

\definecolor{blush}{rgb}{0.87, 0.36, 0.51}
\definecolor{celestialblue}{rgb}{0.29, 0.59, 0.82}
\definecolor{chocolate(web)}{rgb}{0.82, 0.41, 0.12}

\hypersetup{
	colorlinks = true,
	linkcolor = celestialblue,
	anchorcolor = blue,
	citecolor = blush,
	filecolor = blue,
	urlcolor = chocolate(web)
}

\begin{document}
	
\title[Two-weighted estimates of $I_{\gamma,m}$]{Two-weighted estimates of the multilinear fractional integral operator between weighted Lebesgue and Lipschitz spaces with optimal parameters}
	
	%{A characterization of continuity properties of the multilinear fractional integral operator between Lebesgue and Lipschitz spaces
	
	% Information for first author
	\author[F. Berra]{Fabio Berra}
	\address{CONICET and Departamento de Matem\'{a}tica (FIQ-UNL),  Santa Fe, Argentina.}
	\email{fberra@santafe-conicet.gov.ar}
	
	%Information for second author
	\author[G. Pradolini]{Gladis Pradolini}
	\address{CONICET and Departamento de Matem\'{a}tica (FIQ-UNL),  Santa Fe, Argentina.}
	\email{gladis.pradolini@gmail.com}
%	
	%Information for third author
	\author[W. Ramos]{Wilfredo Ramos}
	\address{CONICET and Departamento de Matem\'{a}tica (FaCENA-UNNE),  Corrientes, Argentina.}
	\email{oderfliw769@gmail.com}
	
	\thanks{The author were supported by CONICET, UNL and UNNE.}
	
	\subjclass[2010]{26A33, 42B25}
	
	\keywords{Multilinear fractional operator, Lipschitz spaces, weights}
	
	%-------------ABSTRACT ------------------------------------
	\begin{abstract}
			Given an $m$-tuple of weights $\vec{v}=(v_1,\dots,v_m)$, we characterize the classes of pairs $(w,\vec{v})$ involved with the boundedness properties of the multilinear fractional integral operator from $\prod_{i=1}^mL^{p_i}\left(v_i^{p_i}\right)$ into suitable Lipschitz spaces associated to a parameter $\delta$, $\mathcal{L}_w(\delta)$. Our results generalize some previous estimates not only for the linear case but also for the unweighted problem in the multilinear context. We emphasize the study related to the range of the parameters involved with the problem described above, which is optimal in the sense that they become trivial outside of the region obtained.
		 We also exhibit nontrivial examples of pairs of weights in this region.    
	\end{abstract}

	\maketitle
	
	\section*{}
	
	\medskip
	\medskip
	
	\vspace*{-1.5cm}
	\hrule
	\vspace*{0.1cm}
	\hrule
	\vspace*{0.5cm}
	\begin{center}
		\textit{This article is dedicated to Professor Eleonor ``Pola'' Harboure, beloved colleague whose vast knowledge and human kindness have always been a guidance to us.}
	\end{center}
	
	\vspace*{0.5cm}\hrule\vspace*{0.1cm}\hrule\vspace*{0.75cm}
		\medskip
	\medskip	
	
	\section{Introduction}\label{seccion: introduccion}
	
	In 1972 B. Muckenhoupt characterized the nonnegative functions $w$ for which the classical Hardy-Littlewood maximal operator $M$ is bounded in $L^p(w)$, for $1<p<\infty$ (see \cite{Muck72}). More precisely, the author proved that $M:L^p(w)\hookrightarrow L^p(w)$ if and only if $w\in A_p$, that is $w$ satisfies the inequality
	\[\left(\frac{1}{|Q|}\int_Q w\right)\left(\frac{1}{|Q|}\int_Q w^{1-p'}\right)^{p-1}\leq C\]
	for every cube $Q$.
	These classes became very important for many estimates in Harmonic Analysis and were further studied by many authors.
	
	Later on, in \cite{Muckenhoupt-Wheeden74}, B. Muckenhoupt and R. Wheeden introduced a variant of these sets of functions, the $A_{p,q}$ classes, given by the collection of weights $w$ such that
	\[\left(\frac{1}{|Q|}\int_Q w^q\right)^{1/q}\left(\frac{1}{|Q|}\int_Q w^{-p'}\right)^{1/p'}\leq C,\]
	for every cube $Q$, where $1<p,q<\infty$. These classes played an important role on the boundedness properties of the fractional maximal operator $M_\gamma$, $0<\gamma<n$ and the fractional integral operator $I_\gamma$ given by the expression 
	\[I_\gamma f(x)=\int_{\mathbb{R}^n}\frac{f(y)}{|x-y|^{n-\gamma}}\,dy,\]
	whenever the integral is finite. It was proved in \cite{Muckenhoupt-Wheeden74}  that if $1<p<n/\gamma$ and $1/q=1/p-\gamma/n$, then this operator maps $L^p(w^p)$ into $L^q(w^q)$  if and only if $w\in A_{p,q}$. For the endpoint case $p=n/\gamma$ it was also shown that the operator $I_\gamma$ maps $L^{n/\gamma}(w^{n/\gamma})$ into a weighted version of the bounded mean oscillation spaces $\mathrm{BMO}$ if and only if $w^{-n/(n-\gamma)}\in A_1$. Although the $A_{p,q}$ classes above are a variant of $A_p$, they are intimately related with them. It is well-known that $w\in A_{p,q}$ is equivalent either to  $w^q\in A_{1+q/p'}$ or $w^{-p'}\in A_{1+p'/q}$, (see \cite{Muckenhoupt-Wheeden74}). 
	
	Later on, in \cite{Pradolini01} the author proved that for $n/\gamma\leq p<n/(\gamma-1)^+$ and $\delta=\gamma-n/p$ the operator $I_\gamma$ maps $L^p(w^p)$ into suitable weighted Lipschitz spaces related to the parameter $\delta$. These spaces are a generalization of those introduced in \cite{Muckenhoupt-Wheeden74} which correspond to $\delta=0$. A two-weighted problem it was also studied, giving the optimal parameters for which the associated classes of weights are nontrivial.
	
	In \cite{HSV} E. Harboure, O. Salinas and B. Viviani introduced a newfangle class of weighted Lipschitz spaces wider than those considered in \cite{Pradolini01}. Concretely, they defined the class $\mathcal{L}_w(\delta)$ as the collection of locally integrable functions $f$ such that
	\begin{equation}\label{eq: definicion clase Lipschitz w}
		\sup_{B\subset \mathbb{R}^n}\frac{1}{w^{-1}(B)|B|^{\delta/n}}\int_B|f(x)-f_B|\,dx<\infty.
	\end{equation}  
They characterized the weights involved with the continuity properties of $I_{\gamma}$ acting between $L^{p}(w)$ into $\mathcal{L}_w(\delta)$ for $1<p<n/(\gamma-1)^+$ and $\delta=\gamma-n/p$. The class of weights turned out wider than the corresponding class considered in \cite{Pradolini01}, being the same under certain additional assumptions on the weight. Inspired in that work, a two-weighted problem was also studied in \cite{Prado01cal}. 

	Given $m\in\mathbb{N}$ and $0<\gamma<mn$ the multilinear fractional integral operator of order $m$, $I_{\gamma,m}$, is defined as follows
	\[I_{\gamma,m} \vec{f}(x)=\int_{(\mathbb{R}^n)^m} \frac{\prod_{i=1}^m f_i(y_i)}{(\sum_{i=1}^m|x-y_i|)^{mn-\gamma}}\,d\vec{y},\]
	where $\vec{f}=(f_1,f_2,\dots, f_m)$ and $\vec{y}=(y_1,y_2,\dots, y_m)$, provided the integral is finite.
	
	The continuity properties of  $I_{\gamma,m}$ were studied for several authors. For example,  it was shown in \cite{Moen09} that if $0<\gamma<mn$ then
	$I_{\gamma,m}: \prod_{i=1}^m L^{p_i}\hookrightarrow L^q$, where $1/p=\sum_{i=1}^m1/p_i$ and $1/q=1/p-\gamma/n$. The author also considered weighted versions of these estimates, generalizing the results of \cite{Muckenhoupt-Wheeden74} to the multilinear context. On the other hand, in \cite{AHIV} unweighted estimates of $I_{\gamma,m}$ between $\prod_{i=1}^m L^{p_i}$ and  Lipschitz-$\delta$ spaces were given, with $0\leq \delta<1$ and $\delta=\gamma-n/p$. For other type of estimates involving multilinear version of the fractional integral operator see also \cite{Grafakos92}, \cite{GK01},  \cite{KS99} and \cite{Pradolini10}.

	Recently in \cite{BPR22} we studied the boundedness of $I_{\gamma,m}$ between $\prod_{i=1}^m L^{p_i}\left(v_i^{p_i}\right)$ into the space $\mathbb{L}_w(\delta)$ defined by the collection of locally integrable functions $f$ such that 
	\begin{equation}\label{eq: definicion clase Lipschitz norma inf}
	\sup_{B\subset \mathbb{R}^n}\frac{\|w\mathcal{X}_B\|_\infty}{|B|^{1+\delta/n}}\int_B|f(x)-f_B|\,dx<\infty,
	\end{equation} 
	characterizing the weights involved as those satisfying the condition $\mathbb{H}_m(\vec{p},\gamma,\delta)$ given by
	\begin{equation}\label{eq: clase Hbb(p,gamma,delta) - m}
	\frac{\|w\mathcal{X}_B\|_\infty}{|B|^{(\delta-1)/n}}\prod_{i=1}^m\left(\int_{\mathbb{R}^n} \frac{v_i^{-p_i'}(y)}{(|B|^{1/n}+|x_B-y|)^{(n-\gamma_i+1/m)p_i'}}\,dy\right)^{1/p_i'}\leq C.
	\end{equation} 
	
	The purpose of this article is to study the boundedness of the operator $I_{\gamma,m}$ between a product of weighted Lebesgue spaces into the Lipschitz space $\mathcal{L}_w(\delta)$ defined in (\ref{eq: definicion clase Lipschitz w}). Our result generalizes the linear case when $p>n/\gamma$. We do not only consider related weights, which is an adequate extension of the one-weight estimates in the linear case proved in \cite{HSV}, but also with independent weights exhibiting an extension of the corresponding problem given in \cite{Prado01cal} for $m=1$. We characterize the classes of weights for which the problem described above holds. We also show the optimal range of the parameters involved. The optimality is understood in the sense that the parameters describe certain region in which we can find concrete examples of weights belonging to the class, becoming trivial outside of it. The results obtained in this paper not only extend the results in  \cite{HSV} and \cite{Prado01cal} but also they generalize the unweighted multilinear results proved in \cite{AHIV} .
	
	We shall now introduce the classes of weights and the notation required in order to state our main results. 
	
	Along the manuscript the multilinear parameter will be denoted by $m\in \mathbb{N}$. Let $0<\gamma<mn$, $\delta\in \mathbb{R}$ and $\vec{p}=(p_1,p_2,\dots, p_m)$ be an $m$-tuple of exponents  where $1\le p_i \le \infty$ for $1\le i\le m$. We define $p$ such that $1/p=\sum_{i=1}^{m}1/p_i$. 
	
	We shall be dealing with a wider class of multilinear weights than those satisfying \eqref{eq: clase Hbb(p,gamma,delta) - m}  (see  \cite{BPR22}) and defined as follows. Given the weights $w$, $v_1,\dots, v_m$, if $\vec{v}=(v_1,v_2,\dots,v_m)$ 
	we say that a pair $(w,\vec{v})$ belongs to the class $\mathcal{H}_m(\vec{p},\gamma,\delta)$ if there exists a positive constant $C$ such that the inequality
	\begin{equation}
		\frac{|B|^{1+(1-\delta)/n}}{w^{-1}(B)}\prod_{i=1}^m\left(\int_{\mathbb{R}^n} \frac{v_i^{-p_i'}(y)}{(|B|^{1/n}+|x_B-y|)^{(n-\gamma_i+1/m)p_i'}}\,dy\right)^{1/p_i'}\leq C
	\end{equation} 
	holds for every ball $B=B(x_B, R)$, where $x_B$ denotes the center of $B$ and $\sum_{i=1}^m\gamma_i=\gamma$, with $0<\gamma_i<n$ for every $i$. The integral above is understood as usual when $p_i=1$, (see \S~\ref{section: preliminares} for further details).
	
	When $m=1$ the class given above was first introduced in \cite{Prado01cal} (for $w=v$ see also \cite{MW-75-76} for the case $\delta=0$ and \cite{HSV} for the one-weight case). In that paper the author showed  nontrivial weights when $\delta\leq\min\{1,\gamma-n/p\}$. A similar restriction, as we shall prove, appears in the multilinear context.
	
\begin{obs}\label{obs: Hbb contenida en Hcal}
	It is easy to check that $\mathbb{H}_m(\vec{p},\gamma,\delta)\subset \mathcal{H}_m(\vec{p},\gamma,\delta)$ and, if $w^{-1}\in A_1$, both classes coincide. The same statement is true for the classes $\mathbb{L}_w(\delta)$ and $\mathcal{L}_w(\delta)$.
\end{obs}

We are now in a position to state our main results.

	\begin{teo}\label{teo: teo principal - Hcal}
		Let $0<\gamma<mn$, $\delta\in\mathbb{R}$,  and $\vec{p}$ a vector of exponents that verifies $p>n/\gamma$. Let $(w,\vec{v})$ a pair such that $v_i^{-p_i'}\in \mathrm{RH}_{m}$, for $i\in\mathcal{I}_2=\{1\leq i\leq m: 1<p_i\leq \infty\}$. Then the following statements are equivalent:
		\begin{enumerate}[\rm(1)]
			\item \label{item: teo principal - Hcal item 1} The operator $I_{\gamma,m}$ is bounded from  $\prod_{i=1}^m L^{p_i}(v_i^{p_i})$ to $\mathcal{L}_{w}(\delta)$;
			\item \label{item: teo principal - Hcal item 2} The pair $(w,\vec{v})$ belongs to $\mathcal{H}_m(\vec{p},\gamma,\delta)$. 
		\end{enumerate}
	\end{teo}

Observe that a reverse Hölder condition for the weights $v_i$ is required for our theorem to hold. Although this seems to be a restriction, it does trivially hold when we consider $m=1$, as expected. A condition of this type was also required for the class $\mathbb{H}_m(\vec{p},\gamma,\delta)$ in \cite{BPR22}.

We also notice that whilst there is no restriction on $\delta$ in the previous theorem, they arise as a consequence of the nature of the corresponding weights. The following theorem establishes the range of parameters involved in the class $\mathcal{H}_m(\vec{p},\gamma,\delta)$ where the weights are trivial, that is, $v_i=\infty$ a.e. for some $i$ or $w=0$ a.e. 
	
\begin{teo}\label{teo: no-ejemplos Hcal} Let $0<\gamma<mn$, $\delta\in\mathbb{R}$,  and $\vec{p}$ a vector of exponents. The following statements hold:
	\begin{enumerate}[\rm(a)]
		\item\label{item: teo no-ejemplos Hcal - item a} If $\delta>1$ or $\delta>\gamma-n/p$ then condition $\mathcal{H}_m(\vec{p},\gamma,\delta)$ is satisfied if and only if $v_i=\infty$ a.e. for some $1\le i\le m$.
		\item\label{item: teo no-ejemplos Hcal - item b} The same conclusion holds if $\delta=\gamma-n/p=1$.
		\end{enumerate}
\end{teo}

In \S~\ref{seccion: ejemplos} we shall exhibit non trivial examples of pairs $(w,\vec{v})$, for which the class $\mathcal{H}_m(\vec{p},\gamma,\delta)$ is non empty, depicting the corresponding regions described by the parameters. By Remark~\ref{obs: Hbb contenida en Hcal} we have that these region include the corresponding ones given in \cite{BPR22}. 

Regarding the case when $w=\prod_{i=1}^mv_i$, which generalizes the one-weighted problem when $m=1$, we have proved in \cite{BPR22} that condition $\mathbb{H}_m(\vec{p},\gamma,\delta)$ reduces to the multilinear class $A_{\vec{p},\infty}$. This  is the natural multilinear extension for the condition $v^{-p'}\in A_1$ on the linear setting. When $(w,\vec{v})\in \mathcal{H}_m(\vec{p},\gamma,\delta)$ and $w=\prod_{i=1}^m v_i$ we shall directly say that $\vec{v}\in \mathcal{H}_m(\vec{p},\gamma,\delta)$, that is, there exists a positive constant $C$ such that the inequality
\[|B|^{(1-\delta)/n}\prod_{i=1}^m\left(\int_{\mathbb{R}^n} \frac{v_i^{-p_i'}(y)}{(|B|^{1/n}+|x_B-y|)^{(n-\gamma_i+1/m)p_i'}}\,dy\right)^{1/p_i'}\leq \frac{C}{|B|}\int_B \prod_{i=1}^m v_i^{-1}\]
holds for every ball $B$, with the obvious changes when $p_i=1$ for some $i$. The following theorem deals with this case of related weights.

\begin{teo}\label{teo: caso de pesos iguales}
	Let $0<\gamma<mn$, $\delta\in\mathbb{R}$ and $\vec{p}$ a vector of exponents. If $\vec{v}\in \mathcal{H}_m(\vec{p},\gamma,\delta)$ and $p/(mp-1)>1$ then we have that $\delta=\gamma-n/p$.
\end{teo}

When $m=1$ the theorem above was given in \cite{PR}. As an immediate consequence we have the following result.

\begin{coro}
	Given $0<\gamma<mn$,  $\vec{p}$ a vector of exponents and $\delta=\gamma-n/p$. If $\vec{v}\in \mathcal{H}_m(\vec{p},\gamma,\delta)$ and $\alpha=~p/(mp-1)>1$, then we have that $\prod_{i=1}^mv_i^{-1}\in \mathrm{RH}_\alpha$.
\end{coro}

Notice that, when $m=1$, $\alpha=p'>1$ and the corollary establishes that if $v\in \mathcal{H}_1(p,\gamma,\delta)$ then $v^{-1}\in \mathrm{RH}_{p'}$, a property proved in \cite{HSV}.

\section{Preliminaries and definitions}\label{section: preliminares}

Throughout the paper $C$ will denote an absolute constant that may change in every occurrence. By $A\lesssim B$ we mean that there exists a positive constant $c$ such that $A\leq c B$.  We say that $A\approx B$ when $A\lesssim B$ and $B\lesssim A$. 

Let $m\in \mathbb{N}$. Given a set $E$, with $E^m$ we shall denote the cartesian product of $E$ $m$ times.

It will be useful for us to consider the operator
\begin{equation}\label{eq: operador Jgamma,m}
J_{\gamma,m}\vec{f}(x)=\int_{(\mathbb{R}^n)^m} \left(\frac{1}{(\sum_{i=1}^m|x-y_i|)^{mn-\gamma}}-\frac{1-\mathcal{X}_{B(0,1)^m}(\vec{y})}{(\sum_{i=1}^m|y_i|)^{mn-\gamma}}\right)\prod_{i=1}^m f_i(y_i)\,d\vec{y}.
\end{equation}
which differs from $I_{\gamma,m}$ only by a constant term, therefore it has the same Lipschitz norm as $I_{\gamma,m}$, so it will be enough to give the results for $J_{\gamma,m}$.

By a weight we understand any positive and locally integrable function. As we said in the introduction, given $\delta \in \mathbb{R}$ and a weight $w$ we say that a locally integrable function $f\in  \mathcal{L}_{w}(\delta)$ if there exists a positive constant $C$ such that 
\begin{equation}%\label{eq: definicion clase Lipschitz norma inf}
\frac{1}{w^{-1}(B)|B|^{\delta/n}}\int_B|f(x)-f_B|\,dx\leq C
\end{equation} 
for every ball $B$, where $f_B=|B|^{-1}\int_B f$.

If $\delta=0$ the space $\mathcal{L}_{w}(\delta)$ coincides with some weighted versions of BMO spaces introduced in \cite{MW-75-76}. Concerning to the unweighted case, when $0<\delta<1$ it is equivalent to the classical Lipschitz classes $\Lambda(\delta)$ given by the collection of functions $f$ satisfying $|f(x)-f(y)|\le C |x-y|^{\delta}$ and, if $-n<\delta<0$, they are Morrey spaces. On the other hand, this space was studied for example in \cite{HSV} and in \cite{Prado01cal}. 

The class   $\mathcal{H}_m(\vec{p},\gamma,\delta)$ is given by the pairs $(w,\vec{v})$ for which the inequality
\begin{equation}\label{eq: clase Hcal(p,gamma,delta) - m}
\sup_{B\subset \mathbb{R}^n} \frac{|B|^{1+(1-\delta)/n}}{w^{-1}(B)}\prod_{i=1}^m\left(\int_{\mathbb{R}^n} \frac{v_i^{-p_i'}(y)}{(|B|^{1/n}+|x_B-y|)^{(n-\gamma_i+1/m)p_i'}}\,dy\right)^{1/p_i'}<\infty
\end{equation} 
holds. For those index $i$ such that $p_i=1$ we understand the corresponding factor on the product above as
\begin{equation}\label{eq: factor de H para p_i=1}
\left\|\frac{v_i^{-1}}{(|B|^{1/n}+|x_B-\cdot|)^{(n-\gamma_i+1/m)}}\right\|_\infty.
\end{equation}

 Let $\mathcal{I}_1=\{1\leq i\leq m: p_i=1\}$ and $\mathcal{I}_2=\{1\leq i\leq m: p_i>1\}$. We will also denote with $m_j$  the cardinal of the set $\mathcal{I}_j$, that is, $m_j=\#\mathcal{I}_j$ for $j=1,2$. We shall use this notation throughout the paper.

Observe that if $(w,\vec{v})$ belongs to $\mathcal{H}_m(\vec{p},\gamma,\delta)$, then the inequalities 
\begin{equation}\label{eq: condicion local 2}
	\frac{|B|^{1-\delta/n+\gamma/n-1/p}}{w^{-1}(B)}\prod_{i\in\mathcal{I}_1}\|v_i^{-1}\mathcal{X}_B\|_\infty\,\prod_{i\in\mathcal{I}_2}\left(\frac{1}{|B|}\int_B v_i^{-p_i'}\right)^{1/p_i'}\leq C
\end{equation}
and
\begin{equation}\label{eq: condicion global 2}
	\frac{|B|^{1+(1-\delta)/n}}{w^{-1}(B)}\prod_{i\in\mathcal{I}_1}\left\|\frac{v_i^{-1}\mathcal{X}_{\mathbb{R}^n\backslash B}}{(|B|^{1/n}+|x_B-\cdot|)^{(n-\gamma_i+1/m)}}\right\|_\infty\,\prod_{i\in\mathcal{I}_2}\left(\int_{\mathbb{R}^n\backslash B} \frac{v_i^{-p_i'}(y)}{|x_B-y|^{(n-\gamma_i+1/m)p_i'}}\,dy\right)^{1/p_i'}\leq C,
\end{equation} 
hold for every ball $B$. We shall refer to these inequalities as the \textit{local} and the \textit{global} conditions, respectively. Furthermore, if $\mathcal{I}$ and $\mathcal{J}$ partition the set $\mathcal{I}_1$, from \eqref{eq: clase Hcal(p,gamma,delta) - m} we can write
\begin{equation}\label{eq: local 2 y global 2 mezcladas}
	\frac{|B|^{1+(\gamma-\delta)/n-1/p}}{w^{-1}(B)}\prod_{i\in \mathcal{I}}\left\|v_i^{-1}\mathcal{X}_{2B-B}\right\|_\infty\,\prod_{i\in \mathcal{J}}\left\|v_i^{-1}\mathcal{X}_B\right\|_\infty\,\prod_{i\in\mathcal{I}_2}\left(\frac{1}{|2B|}\int_{2B}v_i^{-p_i'}\right)^{1/p_i'}\leq C
	\end{equation}
	for every ball $B$. This inequality will be useful for our purposes later. 

On the other hand, when $v_i^{-1}\in\mathrm{RH}_\infty$ for $i\in\mathcal{I}_1$ and $v_i^{-p_i'}$ is doubling for $i\in\mathcal{I}_2$, the corresponding local and global conditions  imply \eqref{eq: clase Hcal(p,gamma,delta) - m}. Before state and prove this result, we shall introduce some useful notation.

	Given $m\in\mathbb{N}$ we denote $S_m=\{0,1\}^m$. Given a set $B$ and $\sigma\in S_m$, $\sigma=(\sigma_1,\sigma_2,\dots,\sigma_m)$ we define 
	\[B^{\sigma_i}=\left\{
	\begin{array}{ccl}
	B,&\textrm{ if }&\sigma_i=1\\
	\mathbb{R}^n\backslash B,&\textrm{ if }&\sigma_i=0.
	\end{array}
	\right.\]
	
	With the notation $\mathbf{B}^\sigma$ we will understand the cartesian product $B^{\sigma_1}\times B^{\sigma_2}\times\dots\times B^{\sigma_m}$. Particularly, if we set $\mathbf{1}=(1,1,\dots,1)$ and $\mathbf{0}=(0,0,\dots,0)$ then we have
	\[\mathbf{B}^{\mathbf{1}}=B\times B\times\dots\times B=B^m,\quad\textrm{ and }\quad \mathbf{B}^{\mathbf{0}}=(\mathbb{R}^n\backslash B)\times (\mathbb{R}^n\backslash B)\times\dots\times (\mathbb{R}^n\backslash B)=(\mathbb{R}^n\backslash B)^m.\]
	
\begin{lema}\label{lema: equivalencia con local y global}
Let $0<\gamma<mn$, $\delta\in\mathbb{R}$, $\vec{p}$ a vector of exponents and $(w,\vec{v})$ a pair of weights such that $v_i^{-1}\in\mathrm{RH}_\infty$ for $i\in\mathcal{I}_1$ and $v_i^{-p_i'}$ is doubling for $i\in\mathcal{I}_2$. Then condition $\mathcal{H}_m(\vec{p},\gamma,\delta)$ is equivalent to \eqref{eq: condicion global 2}.
\end{lema}

\begin{proof}
	We have already seen that $\mathcal{H}_m(\vec{p},\gamma,\delta)$ implies  \eqref{eq: condicion global 2}. In order to prove the converse, we let $\theta_i=n-\gamma_i+1/m$, for every $i$. Recall that $m_2=\#\mathcal{I}_2$. After a possible rearrangement of the indices $i\in\mathcal{I}_2$ we have  that
	\[\prod_{i\in\mathcal{I}_2}\left(\int_{\mathbb{R}^n} \frac{v_i^{-p_i'}}{(|B|^{1/n}+|x_B-\cdot|)^{\theta_ip_i'}}\right)^{1/p_i'}=\sum_{\sigma\in S_{m_2}}\prod_{i=1}^{m_2}\left( \int_{B^{\sigma_i}}\frac{v_i^{-p_i'}}{(|B|^{1/n}+|x_B-\cdot|)^{\theta_ip_i'}}\right)^{1/p_i'}.\]
	
	Fix $\sigma\in S_{m_2}$. If $\sigma_i=0$, we have that
	\begin{align*}
	\left(\int_{B^{\sigma_i}}\frac{v_i^{-p_i'}}{(|B|^{1/n}+|x_B-\cdot|)^{(n-\gamma_i+1/m)p_i'}}\right)^{1/p_i'}&=\left(\int_{\mathbb{R}^n\backslash B}\frac{v_i^{-p_i'}}{(|B|^{1/n}+|x_B-\cdot|)^{(n-\gamma_i+1/m)p_i'}}\right)^{1/p_i'}\\
	&\leq \left(\int_{\mathbb{R}^n\backslash B}\frac{v_i^{-p_i'}(y)}{|x_B-y|^{(n-\gamma_i+1/m)p_i'}}\,dy\right)^{1/p_i'}.
	\end{align*}
	
	For $\sigma_i=1$, since $v_i^{-p_i'}$ is doubling, we have that 
	\begin{align*}
	\left(\int_{B^{\sigma_i}}\frac{v_i^{-p_i'}(y)}{(|B|^{1/n}+|x_B-y|)^{(n-\gamma_i+1/m)p_i'}}\,dy\right)^{1/p_i'}&=\left(\int_B\frac{v_i^{-p_i'}(y)}{(|B|^{1/n}+|x_B-y|)^{(n-\gamma_i+1/m)p_i'}}\,dy\right)^{1/p_i'}\\
	&\leq \frac{1}{|B|^{1-\gamma_i/n+1/(mn)}}\left(\int_B v_i^{-p_i'}\right)^{1/p_i'}\\
	& \lesssim \frac{1}{|2B|^{1-\gamma_i/n+1/(mn)}}\left(\int_{2B\backslash B} v_i^{-p_i'}\right)^{1/p_i'}\\
	&\leq\left(\int_{2B\backslash B}\frac{v_i^{-p_i'}(y)}{|x_B-y|^{(n-\gamma_i+1/m)p_i'}}\right)^{1/p_i'}\\
	&\leq\left(\int_{\mathbb{R}^n\backslash B}\frac{v_i^{-p_i'}(y)}{|x_B-y|^{(n-\gamma_i+1/m)p_i'}}\right)^{1/p_i'}.
	\end{align*}
	Therefore, for every $\sigma\in S_{m_2}$ we obtain
	\begin{equation}\label{eq: lema: equivalencia con local y global - eq1}
		\prod_{i=1}^{m_2}\left( \int_{B^{\sigma_i}}\frac{v_i^{-p_i'}}{(|B|^{1/n}+|x_B-\cdot|)^{\theta_ip_i'}}\right)^{1/p_i'}\lesssim \prod_{i\in\mathcal{I}_2}\left(\int_{\mathbb{R}^n\backslash B}\frac{v_i^{-p_i'}}{|x_B-\cdot|^{\theta_ip_i'}}\right)^{1/p_i'}.
	\end{equation}

	On the other hand, for $i\in\mathcal{I}_1$ we proceed similarly as above replacing $\|\cdot\|_{p_i'}$ by $\|\cdot\|_\infty$ and using the $\mathrm{RH}_\infty$ condition for $v_i^{-1}$. Indeed, observe that
	\[\left\|v^{-1}\mathcal{X}_B\right\|_\infty\leq \frac{C}{|B|}\int_B v^{-1}\leq \frac{C}{|B|}\int_{2B\backslash B} v^{-1}\leq C\left\|v^{-1}\mathcal{X}_{2B\backslash B}\right\|_\infty.\]
	Then we can conclude that
	\begin{equation}\label{eq: lema: equivalencia con local y global - eq2}
		\prod_{i\in\mathcal{I}_1} \left\|\frac{v_i^{-1}}{(|B|^{1/n}+|x_B-\cdot|)^{n-\gamma/m+1/m}}\right\|_\infty\lesssim \prod_{i\in\mathcal{I}_1} \left\|\frac{v_i^{-1}\mathcal{X}_{\mathbb{R}^n\backslash B}}{|x_B-\cdot|^{n-\gamma/m+1/m}}\right\|_\infty.
	\end{equation}
	Therefore, by combining \eqref{eq: lema: equivalencia con local y global - eq1}, \eqref{eq: lema: equivalencia con local y global - eq2} and \eqref{eq: condicion global 2} we get that
	\[\frac{|B|^{1+(1-\delta)/n}}{w^{-1}(B)}\prod_{i\in\mathcal{I}_1}\left\|\frac{v_i^{-1}}{(|B|^{1/n}+|x_B-\cdot|)^{n-\gamma/m+1/m}}\right\|_\infty\,\prod_{i\in\mathcal{I}_2}\left(\int_{\mathbb{R}^n} \frac{v_i^{-p_i'}}{(|B|^{1/n}+|x_B-\cdot|)^{\theta_ip_i'}}\right)^{1/p_i'}\leq C,\]
	as desired.
\end{proof}

\begin{coro}
	Under the hypotheses of Lemma~\ref{lema: equivalencia con local y global} we have that condition \eqref{eq: condicion global 2} implies  \eqref{eq: condicion local 2}. 
\end{coro}

\section{Technical results}\label{section: resultados auxiliares}

We now introduce some operators related to $I_{\gamma,m}$ and some useful properties in order to prove our main results. 

Given a ball $B=B(x_B,R)$ and $\tilde B=2B$, as in \cite{BPR22} we can decompose the operator in \eqref{eq: operador Jgamma,m} as
\[J_{\gamma,m}\vec{f}(x)=a_B + I\vec{f}(x),\] where
\begin{equation}\label{eq: definicion de a_B}
	a_B=\int_{(\mathbb{R}^n)^m} \left(\frac{1-\mathcal{X}_{\tilde B^m}(\vec y)}{(\sum_{i=1}^m|x_B-y_i|)^{mn-\gamma}}-\frac{1-\mathcal{X}_{B(0,1)^m}(\vec{y})}{(\sum_{i=1}^m|y_i|)^{mn-\gamma}}\right)\prod_{i=1}^m f_i(y_i)\,d\vec{y}
\end{equation}
and
\begin{equation}\label{eq: definicion de I}
	I\vec{f}(x)=\int_{(\mathbb{R}^n)^m} \left(\frac{1}{(\sum_{i=1}^m |x-y_i|)^{mn-\gamma}}-\frac{1-\mathcal{X}_{\tilde B^m}(\vec{y})}{(\sum_{i=1}^m|x_B-y_i|)^{mn-\gamma}}\right)\prod_{i=1}^m f_i(y_i)\,d\vec{y}.
\end{equation}

We shall first prove that this operator is well-defined for $\vec{f}$ as in Theorem~\ref{teo: teo principal - Hcal}. 

We recall that a weight $w$ belongs to the \textit{reverse H\"{o}lder} class $\mathrm{RH}_s$, $1<s<\infty$, if there exists a positive constant $C$ such that the inequality
\[\left(\frac{1}{|B|}\int_B w^s\right)^{1/s}\leq \frac{C}{|B|}\int_B w\]
holds for every ball $B$ in $\mathbb{R}^n$. It is not difficult to see that $\mathrm{RH}_t\subset \mathrm{RH}_s$ whenever $1<s<t$. We also consider weights belonging to the class $\mathrm{RH}_{\infty}$, that is, the collection of weights $w$ such that the inequality 
\[\sup_B w\le  \frac{C}{|B|}\int_B w,\]
holds for some positive constant $C$.

The next lemma establishes the well definition of $J_{\gamma,m}\vec{f}$, for  $\vec{f}$ as in Theorem~\ref{teo: teo principal - Hcal}. 

	\begin{lema}\label{lema: finitud de J_gamma,m para Hcal}
		Let $0<\gamma<mn$, $\delta\in\mathbb{R}$,  and $\vec{p}$ a vector of exponents that verifies $p>n/\gamma$. Let $(w,\vec{v})$ be a pair of weights in $\mathcal{H}_m(\vec{p},\gamma,\delta)$ such that $v_i^{-p_i'}\in \mathrm{RH}_{m}$, for $i\in\mathcal{I}_2$. If $\vec{f}$ satisfies $f_iv_i\in L^{p_i}$ for every $1\leq i\leq m$, then $J_{\gamma,m}\vec{f}$ is finite in almost every $x\in \mathbb{R}^n$.
	\end{lema}

\begin{proof}
	We are going to exhibit a sketch of the proof, since it follows similar lines to that in \cite{BPR22}, Lemma 3.1. By using the same notation as in that lemma, fix a ball $B=B(x_B, R)$ and write $J_{\gamma,m}\vec{f}=a_B+I\vec{f}$, where we  split $a_B=a_B^1+a_B^2$ and $I\vec{f}=I_1\vec{f}+I_2\vec{f}$.
	We proved that
	\[|a_B^1|\leq \left(1+\frac{C}{|B|^{m-\gamma/n}}\right)\prod_{i=1 }^m\left(\int_{B_0}|f_i(y_i)|\,d_{y_i}\right),\]
	where $B_0=B(0,R_0)$ with $R_0=2(|x_B|+R)$.
	By using H\"{o}lder inequality and condition \eqref{eq: condicion local 2} we get
	\begin{align*}
	|a_B^1|&\leq \left(1+\frac{C}{|B|^{m-\gamma/n}}\right)	\prod_{i=1}^m\|f_iv_i\|_{p_i}\prod_{i\in \mathcal{I}_1} \left\|v_i^{-1}\mathcal{X}_{B_0}\right\|_\infty\,\prod_{i\in \mathcal{I}_2}\left(\int_{B_0}v_i^{-p_i'}\right)^{1/p_i'}\\
	&\leq \left(1+\frac{C}{|B|^{m-\gamma/n}}\right)\prod_{i=1}^m\|f_iv_i\|_{p_i}\frac{w^{-1}(B_0)}{|B_0|}|B_0|^{\delta/n-\gamma/n+1/p}\\
	&<\infty.
	\end{align*}
	In the same lemma we also proved that 
	\[|a_B^2|\leq C \prod_{i=1 }^m\|f_iv_i\|_{p_i}\prod_{i\in \mathcal{I}_1} \left\|\frac{v_i^{-1}}{(|B_0|^{1/n}+|x_{B_0}-\cdot|)^{\theta_i}}\right\|_\infty\,\prod_{i\in \mathcal{I}_2}\left(\int_{\mathbb{R}^n}\frac{v_i^{-p_i'}}{(|B_0|^{1/n}+|x_{B_0}-y_i|)^{\theta_i p_i'}}\right)^{1/p_i'},\]
	where $\theta_i=n-\gamma_i+1/m$. So by using condition \eqref{eq: clase Hcal(p,gamma,delta) - m} we get that
	\[|a_B^2|\leq C\frac{w^{-1}(B_0)}{|B_0|}|B_0|^{(\delta-1)/n}\prod_{i=1 }^m\|f_iv_i\|_{p_i}<\infty.\]
	Let us now consider  $I_1\vec{f}$. By proceeding as in the corresponding estimate in \cite{BPR22} we obtain
	\begin{align*}
	\int_B|I_1\vec{f}(x)|\,dx&\leq C\prod_{i=1}^m\|f_iv_i\|_{p_i}\prod_{i\in \mathcal{I}_2}\left(\frac{1}{|\tilde B|}\int_{\tilde B}v_i^{-p_i'}\right)^{1/p_i'}\times\\
	&\quad\times \prod_{i\in \mathcal{I}_1}\left\| v_i^{-1}\mathcal{X}_{\tilde B}\right\|_\infty|\tilde B|^{(\gamma-\gamma_0)/n-m_1+1/q'+1-1/(m_0p^*)}\\
	&=C|\tilde B|^{\gamma/n-1/p+1}\prod_{i=1}^m\|f_iv_i\|_{p_i}\prod_{i\in \mathcal{I}_1}\left\| v_i^{-1}\mathcal{X}_{\tilde B}\right\|_\infty\,\prod_{i\in \mathcal{I}_2}\left(\frac{1}{|\tilde B|}\int_{\tilde B}v_i^{-p_i'}\right)^{1/p_i'}.
	\end{align*}
	
	We rearrange the indices in $\mathcal{I}_1$ increasingly, in a way to get $\mathcal{I}_1=\{i_1,\dots,i_{m_1}\}$. Observe that
	\[\prod_{i\in\mathcal{I}_1}\left\|v_i^{-1}\mathcal{X}_{\tilde B}\right\|_\infty\leq \prod_{i\in\mathcal{I}_1}\left(\left\|v_i^{-1}\mathcal{X}_{\tilde B-B}\right\|_\infty+\left\|v_i^{-1}\mathcal{X}_{B}\right\|_\infty\right)=\sum_{\sigma\in S^{m_1}}\prod_{j=1}^{m_1}\left\|v_{i_j}^{-1}\mathcal{X}_{\tilde B-B}\right\|_\infty^{\sigma_j}\left\|v_{i_j}^{-1}\mathcal{X}_{B}\right\|_\infty^{1-\sigma_j}.\]
	Therefore,
	\begin{align*}
\int_B |I_1\vec{f}(x)|\,dx&\leq C\left(\prod_{i=1}^m\|f_iv_i\|_{p_i}\right)\times\\
&\quad \times \sum_{\sigma\in S^{m_1}}|\tilde B|^{\gamma/n-1/p+1}\prod_{i\in \mathcal{I}_2}\left(\frac{1}{|\tilde B|}\int_{\tilde B}v_i^{-p_i'}\right)^{1/p_i'}\prod_{j=1}^{m_1}\left\|v_{i_j}^{-1}\mathcal{X}_{\tilde B-B}\right\|_\infty^{\sigma_j}\left\|v_{i_j}^{-1}\mathcal{X}_{B}\right\|_\infty^{1-\sigma_j}
	\end{align*}
	
	 Fix $\sigma\in S^{m_1}$ and define the sets
	\[\mathcal{I}=\{i_j\in\mathcal{I}_1: \sigma_j=1\} \quad\textrm{ and }\quad \mathcal{J}=\{i_j\in\mathcal{I}_1: \sigma_j=0\}.\]
	We can apply condition \eqref{eq: local 2 y global 2 mezcladas}  to bound every term of the sum by
	\[C\frac{w^{-1}(B)}{|B|^{1+(\gamma-\delta)/n-1/p}}|\tilde B|^{\gamma/n-1/p+1}=Cw^{-1}(B)|B|^{\delta/n}.\]
	Consequently,
	\[\int_B |I_1\vec{f}(x)|\,dx\leq Cw^{-1}(B)|B|^{\delta/n}\left(\prod_{i=1}^m\|f_iv_i\|_{p_i}\right).\]
	
	Finally, for $I_2\vec{f}$ we have
	\[|I_2\vec{f}(x)|\leq |B|^{1/n}\sum_{\sigma\in S_m,\sigma\neq \mathbf{1}} \int_{\mathbf{\tilde B}^\sigma} \frac{\prod_{i=1}^m|f_i(y_i)|}{(\sum_{i=1}^m|x_B-y_i|)^{mn-\gamma+1}}\,d\vec{y}.\]
	This expression is similar to $a_B^2$, with $B_0$ replaced by $\tilde B$. Observe that
	\[\left\|\frac{v_i^{-1}}{|x_B-\cdot|^{\theta_i}}\mathcal{X}_{\tilde B^c}\right\|_\infty\leq \left\|\frac{v_i^{-1}}{|x_B-\cdot|^{\theta_i}}\mathcal{X}_{B^c}\right\|_\infty\]
	for those indices $i\in\mathcal{I}_1$ such that $\sigma_i=0$. On the other hand, if $i\in\mathcal{I}_1$ and $\sigma_i=1$, we can split the expression $\|v_i^{-1}\mathcal{X}_{\tilde B}\|_\infty$ as follows
	\[\left\|v_i^{-1}\mathcal{X}_{\tilde B}\right\|_\infty\leq \left\|v_i^{-1}\mathcal{X}_{\tilde B-B}\right\|_\infty+\left\|v_i^{-1}\mathcal{X}_{B}\right\|_\infty\]
	and repeat the argument used in the estimation of $I_1\vec{f}$. After applying condition \eqref{eq: local 2 y global 2 mezcladas} we get that
	\[\int_B|I_2\vec{f}(x)|\,dx\leq Cw^{-1}(B)|B|^{\delta/n}\prod_{i=1}^m\|f_iv_i\|_{p_i}.\]
	This concludes the proof of the lemma.\qedhere
\end{proof}
\begin{obs}
The corresponding bound obtained for $I\vec{f}$ will be used for the proof of Theorem~\ref{teo: teo principal - Hcal}. 
\end{obs}

The next lemma was given in \cite{BPR22}. The sets involved in its statement are defined as follows. 

For a fixed ball $B=B(x_B,R)$ we set 
\[A=\{x_B+h: h=(h_1,h_2,\dots,h_n): h_i\geq 0 \textrm{ for }1\leq i\leq n\},\]
\[C_1=B\left(x_B-\frac{R}{12\sqrt{n}}u,\frac{R}{12\sqrt{n}}\right)\cap\left\{x_B-\frac{R}{12\sqrt{n}}u+h: h_i\leq 0 \textrm{ for every }i\right\},\]
and
\[C_2=B\left(x_B-\frac{R}{3\sqrt{n}}u,\frac{2R}{3}\right)\cap\left\{x_B-\frac{R}{3\sqrt{n}}u+h: h_i\leq 0 \textrm{ for every }i\right\},\]
where $u=(1,1,\dots,1)$. 

\begin{lema}\label{lema: diferencia de nucleos positiva}
	There exists a positive constant $C=C(n)$ such that the inequality
	\[\frac{1}{(\sum_{j=1}^m|x-y_j|)^{mn-\gamma}}-\frac{1}{(\sum_{j=1}^m|z-y_j|)^{mn-\gamma}}\geq C\frac{|B|^{1/n}}{(|B|^{1/n}+\sum_{j=1}^m|x_B-y_j|)^{mn-\gamma+1}}\]
	holds for every $x\in C_1$, $z\in C_2$, and $y_j\in A$ for $1\leq j\leq m$.
\end{lema}

\begin{obs}\label{obs: medida de conjuntos C como B}
	It is not difficult to see that $|C_i|\approx |B|$, for $i=1,2$.
\end{obs}

\section{Proof of the main results}\label{section: prueba principal}

In this section we prove our main results.

\begin{proof}[Proof of Theorem~\ref{teo: teo principal - Hcal}]
	We shall first prove that $(\ref{item: teo principal - Hcal item 2})$ implies $(\ref{item: teo principal - Hcal item 1})$. We shall deal with the operator $J_{\gamma,m}$ since it differs from $I_{\gamma,m}$ by a constant term. We want to prove that for every ball $B$
	\begin{equation}\label{eq: teo principal - eq1}
	\frac{1}{w^{-1}(B)|B|^{\delta/n}}\int_B |J_{\gamma,m}\vec{f}(x)-(J_{\gamma,m}\vec{f})_B|\,dx\leq C\prod_{i=1}^m\|f_iv_i\|_{p_i},
	\end{equation}
	with $C$ independent of $B$. Fix a ball $B=B(x_B,R)$ and recall that $J_{\gamma,m}\vec{f}(x)=a_B+I\vec{f}(x)$. In  Lemma~\ref{lema: finitud de J_gamma,m para Hcal} we proved that
	\begin{equation*}
	\int_B|I\vec{f}(x)|\,dx\leq Cw^{-1}(B)|B|^{\delta/n}\prod_{i=1}^m\|f_iv_i\|_{p_i},
	\end{equation*}
	which implies that
	\begin{equation}\label{eq: teo principal - estimacion de If}
	\int_B|J_{\gamma,m}\vec{f}(x)-a_B|\,dx\leq Cw^{-1}(B)|B|^{\delta/n}\prod_{i=1}^m\|f_iv_i\|_{p_i}.
	\end{equation}
	On the other hand, observe that
	\begin{align*}
	\int_B |J_{\gamma,m}\vec f (x)-(J_{\gamma,m}\vec{f}\,)_B|\,dx&\leq \int_B|J_{\gamma,m}\vec{f}(x)-a_B|\,dx+\int_B|(J_{\gamma,m}\vec f\,)_B-a_B|\,dx\\
	&\leq \int_B|J_{\gamma,m}\vec{f}(x)-a_B|\,dx+\int_B\frac{1}{|B|}\int_B|J_{\gamma,m}\vec f(y)-a_B|\,dy\,dx\\
	&\leq 2\int_B|J_{\gamma,m}\vec{f}(x)-a_B|\,dx.
	\end{align*}
	By combining this estimate with \eqref{eq: teo principal - estimacion de If} we obtain the desired inequality.
	
	We now prove that $(\ref{item: teo principal - Hcal item 1})$ implies $(\ref{item: teo principal - Hcal item 2})$. Assume that the component functions $f_i$ of $\vec{f}$ are nonnegative. We have that \eqref{eq: teo principal - eq1} holds for every ball $B=B(x_B,R)$. Also observe that
	\[\frac{1}{|B|}\int_B|g(x)-g_B|\,dx\approx \frac{1}{|B|^2}\int_B\int_B|g(x)-g(z)|\,dx\,dz,\]
	and therefore the left hand side of \eqref{eq: teo principal - eq1} is equivalent to
	\[\frac{1}{w^{-1}(B)|B|^{1+\delta/n}}\int_B\int_B |J_{\gamma,m}\vec{f}(x)-J_{\gamma,m}\vec{f}(z)|\,dx\,dz=I.\]
	Observe that, when $y_i\in B$ for every $i$ we have  
	\[|B|^{1/n}+|x_B-y_j|\geq \frac{1}{m}\left(|B|^{1/n}+\sum_{i=1}^m|x_B-y_i|\right),\]
	for every $1\leq j\leq m$.
	By combining Lemma~\ref{lema: diferencia de nucleos positiva} and Remark~\ref{obs: medida de conjuntos C como B} with the inequality above we can estimate $I$ as follows
	\begin{align*}
	I&\geq \frac{1}{w^{-1}(B)|B|^{1+\delta/n}}\int_{C_2}\int_{C_1} \int_{A^m} \frac{|B|^{1/n}\prod_{i=1}^m f_i(y_i)}{(|B|^{1/n}+\sum_{i=1}^m|x_B-y_i|)^{mn-\gamma+1}}\,d\vec{y}\,dx\,dz\\
	&\geq C\frac{|B|^{1+(1-\delta)/n}}{w^{-1}(B)}\prod_{i=1}^m\left(\int_A \frac{f_i(y_i)}{(|B|^{1/n}+|x_B-y_i|)^{n-\gamma_i+1/m}}\,dy_i\right).
	\end{align*}
	Since the set $A$ is a quadrant from $x_B$, a similar estimation can be obtained for the other quadrants from $x_B$. Thus, we get
	\[I\geq C\frac{|B|^{1+(1-\delta)/n}}{w^{-1}(B)}\prod_{i=1}^m\left(\int_{\mathbb{R}^n} \frac{f_i(y)}{(|B|^{1/n}+|x_B-y|)^{n-\gamma_i+1/m}}\,dy\right),\]
	which implies that
	\begin{equation}\label{eq: teo principal - eq2}
	\frac{|B|^{1+(1-\delta)/n}}{w^{-1}(B)}\prod_{i=1}^m\left(\int_{\mathbb{R}^n} \frac{f_i(y)}{(|B|^{1/n}+|x_B-y|)^{n-\gamma_i+1/m}}\,dy\right)\leq C\prod_{i=1}^m\|f_iv_i\|_{p_i}.
	\end{equation}
	For every $i\in \mathcal{I}_1$ and $k\in\mathbb{N}$ we define $V_k^i=\{x: v_i^{-1}(x)\leq k\}$ and the functionals
	\[F_i^k(g)=\int_{\mathbb{R}^n}\frac{g(y)v_i^{-1}(y)\mathcal{X}_{V_k^i}(y)}{(|B|^{1/n}+|x_B-y|)^{n-\gamma_i+1/m}}\,dy.\] 
	Therefore $F_i^k$ is a functional in $(L^1)^*=L^{\infty}$. Indeed, if $g\in L^1$
	\[|F_i^k(g)|\leq \|g\|_{L^1} \left\|\frac{v_i^{-1}\mathcal{X}_{V_k^i}}{(|B|^{1/n}+|x_B-\cdot|)^{n-\gamma_i+1/m}}\right\|_\infty<\infty,\]
	and we also get
	\[\frac{|F_i^k(f_iv_i)|}{\|f_iv_i\|_{L^1}}\leq \left\|\frac{v_i^{-1}\mathcal{X}_{V_k^i}}{(|B|^{1/n}+|x_B-\cdot|)^{n-\gamma_i+1/m}}\right\|_\infty,\]
	for every $i\in\mathcal{I}_1$.
	
	If $i\in \mathcal{I}_2$ then we set $A_k=A\cap B(0,k)$ and consider 
	\[f_i^k (y)=\frac{v_i^{-p_i'}(y)}{(|B|^{1/n}+|x_B-y|)^{(n-\gamma_i+1/m)/(p_i-1)}}\mathcal{X}_{A_k}(y)\mathcal{X}_{V_k^i}(y).\]
	
	Let us choose $\vec f=(f_1,\dots,f_m)$, where $f_iv_i\in L^1$ for $p_i=1$ and $f_i=f_i^k$ for $p_i>1$, for fixed $k$. Therefore, the left hand side of \eqref{eq: teo principal - eq2} can be written as follows 
	\[\frac{|B|^{1+(1-\delta)/n}}{w^{-1}(B)}\prod_{i\in \mathcal{I}_1}F_i^k(f_iv_i)\prod_{i\in \mathcal{I}_2}\left(\int_{A_k\cap V_k^i} \frac{v_i^{-p_i'}(y)}{(|B|^{1/n}+|x_B-y|)^{(n-\gamma_i+1/m)p_i'}}\,dy\right)\]
	and it is bounded by
	\[ C\prod_{i\in\mathcal{I}_1}\|f_iv_i\|_{L^1}\prod_{i\in\mathcal{I}_2}\left(\int_{A_k\cap V_k^i} \frac{v_i^{-p_i'}(y)}{(|B|^{1/n}+|x_B-y|)^{(n-\gamma_i+1/m)p_i'}}\,dy\right)^{1/p_i}.\]
	This yields
	\[\frac{|B|^{1+(1-\delta)/n}}{w^{-1}(B)}\prod_{i\in\mathcal{I}_1}\frac{|F_i^k(f_iv_i)|}{\|f_iv_i\|_{L^{1}}}\prod_{i\in\mathcal{I}_2}\left(\int_{A_k\cap V_k^i} \frac{v_i^{-p_i'}(y)}{(|B|^{1/n}+|x_B-y|)^{(n-\gamma_i+1/m)p_i'}}\,dy\right)^{1/p_i'}\leq C,\]
	for every nonnegative $f_i$ such that $f_iv_i\in L^1$, $i\in\mathcal{I}_1$ and for every $k\in\mathbb{N}$. By taking the supremum over these $f_i$ we get
	\begin{align*}
	&\frac{|B|^{1+(1-\delta)/n}}{w^{-1}(B)}\prod_{i\in\mathcal{I}_1}\left\|\frac{v_i^{-1}}{(|B|^{1/n}+|x_B-\cdot|)^{n-\gamma_i+1/m}}\right\|_\infty\,\prod_{i\in\mathcal{I}_2}\left(\int \frac{v_i^{-p_i'}\mathcal{X}_{A_k\cap V_k^i}}{(|B|^{1/n}+|x_B-\cdot|)^{(n-\gamma_i+1/m)p_i'}}\right)^{\tfrac{1}{p_i'}}\\
	&\qquad\leq C.
	\end{align*}
	By taking limit for $k\to\infty$, the left hand side converges to
	\[\frac{|B|^{1+(1-\delta)/n}}{w^{-1}(B)}\prod_{i\in\mathcal{I}_1}\left\|\frac{v_i^{-1}}{(|B|^{1/n}+|x_B-\cdot|)^{n-\gamma_i+1/m}}\right\|_\infty\,\prod_{i\in\mathcal{I}_2}\left(\int_{\mathbb{R}^n} \frac{v_i^{-p_i'}(y)}{(|B|^{1/n}+|x_B-y|)^{(n-\gamma_i+1/m)p_i'}}\,dy\right)^{\tfrac{1}{p_i'}}\]
	which is precisely the condition $\mathcal{H}_m(\vec{p},\gamma,\delta)$. This completes the proof.\qedhere
\end{proof}
	
	\medskip
	
	\begin{proof}[Proof of Theorem~\ref{teo: no-ejemplos Hcal}]
		We begin with item~\eqref{item: teo no-ejemplos Hcal - item a}.  We shall first assume that 
		$\delta >1$.  If $(w,\vec{v}) \in \mathcal{H}_m(\vec{p},\gamma,\delta)$, we choose $B=B(x_B, R)$ where $x_B$ is a Lebesgue point of $w^{-1}$. From \eqref{eq: clase Hcal(p,gamma,delta) - m} we obtain 
		\[\prod_{i\in\mathcal{I}_1}\left\|\frac{v_i^{-1}}{(|B|^{1/n}+|x_B-\cdot|)^{n-\gamma_i+1/m}}\right\|_\infty\,\prod_{i\in\mathcal{I}_2}\left(\int_{\mathbb{R}^n}\frac{v_i^{-p_i'}}{(|B|^{1/n}+|x_B-\cdot|)^{(n-\gamma_i+1/m)p_i'}}\right)^{\tfrac{1}{p_i'}}
		\lesssim \frac{w^{-1}(B)}{|B|R^{1-\delta}},\]
		for every $R>0$. By letting $R\to 0$ and applying the monotone convergence theorem, we conclude that at least one limit factor in the product should be zero. That is, there exists $1\leq i \leq m$ such that $v_i = \infty$ almost everywhere. 
		
		On the other hand, if $\delta > \gamma - n/p$ and $(w,\vec{v})$ belongs to  $\mathcal{H}_m(\vec{p},\gamma,\delta)$, we pick a ball $B=~B(x_B, R)$, where $x_B$ is a Lebesgue point of $w^{-1}$ and every $v_i^{-1}$. Then, by applying \eqref{eq: condicion local 2}, we have
		\begin{align*}
		\prod_{i=1}^m \frac{1}{|B|}\int_B v_i^{-1}&\leq \prod_{i\in\mathcal{I}_1}\left\|v_i^{-1}\mathcal{X}_B\right\|_\infty\,\prod_{i\in\mathcal{I}_2}\left( \frac{1}{|B|}\int_B v_i^{-p'_i} \right)^{1/p'_i }\\
		&\leq C \frac{w^{-1}(B)}{|B|} R^{\delta - \gamma + n/p}
		\end{align*}	
		for every $R>0$. By letting $R\to 0$ we get
		\begin{equation*}
		\prod_{i=1 }^{m} v_{i}^{-1}(x_B)=0, 
		\end{equation*}	
		which yields that $\prod_{i=1 }^{m} v_{i}^{-1}$ is zero almost everywhere. This implies that $M=\bigcap_{i=1}^m \{v_i^{-1} >0 \}$ has null measure. Since $v_i(y)>0$ for almost every $y$ and every $i$, there exists $j$ such that $v_j = \infty$ almost everywhere.
		
		We turn now our attention to item \eqref{item: teo no-ejemplos Hcal - item b}, that is,  $\delta= \gamma - n/p =1$. We shall prove that if $(w,\vec{v})\in~\mathcal{H}_m(\vec{p}, \gamma, 1)$, there exists $j$ such that $v_j = \infty$ in almost $\mathbb{R}^n$. We define
		\begin{equation*}
		\frac{1}{\alpha} = \sum_{i=1 }^{m }\frac{1}{p'_i} = \frac{mp-1}{p}.
		\end{equation*}
		By applying Hölder inequality we obtain that
		\begin{equation*}
		\left(\int_{\mathbb{R}^n } \frac{(\prod_{i\in\mathcal{I}_2} v_i^{-1 })^{\alpha}}{(|B|^{1/n} + |x_B -y|)^{\sum_{i\in\mathcal{I}_2}(n-\gamma_i +1/m)\alpha}} \right)^{1/\alpha} 
		\leq C \prod_{i\in\mathcal{I}_2} 
		\left(
		\int_{\mathbb{R}^n } \frac{ v_i^{-p'_i }}{(|B|^{1/n} + |x_B -\cdot|)^{(n-\gamma_i + 1/m)p'_i} } \right)^{\tfrac{1}{p_i'}}
		\end{equation*}
		and since $(w,\vec{v})\in \mathcal{H}_{m}(\vec{p}, \gamma, 1)$ this implies that
		\[\prod_{i\in \mathcal{I}_1}\left\|\frac{v_i^{-1}}{(|B|^{1/n}+|x_B-\cdot|)^{n-\gamma_i+1/m}}\right\|_\infty\left(\int_{\mathbb{R}^n } \frac{(\prod_{i\in\mathcal{I}_2} v_i^{-1 })^{\alpha}}{(|B|^{1/n} + |x_B -y|)^{\sum_{i\in\mathcal{I}_2}(n-\gamma_i +1/m)\alpha }} \right)^{1/\alpha}\lesssim \frac{w^{-1}(B)}{|B|},\]
		and furthermore
		\[\left(\int_{\mathbb{R}^n } \frac{(\prod_{i=1}^m v_i^{-1 })^{\alpha}}{(|B|^{1/n} + |x_B -y|)^{(mn-\gamma +1)\alpha }}\right) ^{1/\alpha}\lesssim \frac{w^{-1}(B)}{|B|}\]
		for every ball $B=B(x_B,R)$.
		
		If we assume that the set $E=\{x: ~ \prod_{i=1 }^m v_i^{-1 }(x)>~0\}$ has positive measure, we arrive to a contradiction by following the same argument as in Theorem 1.2, item (b) from \cite{BPR22}. This yields  
		$|E|=0$, that is, $\prod_{i=1 }^m v_i^{-1} =0$ almost everywhere, from where we can deduce that there exists  an index $j$ satisfying $v_j = \infty$ almost everywhere.
	\end{proof}

\section{The class \texorpdfstring{$\mathcal{H}_m(\vec{p},\gamma,\delta)$}{$Hm(p,\gamma,\delta)$}}\label{seccion: ejemplos}

 We begin this section by exhibiting nontrivial pairs of weights satisfying condition $\mathcal{H}_m(\vec{p},\gamma,\delta)$. Concretely, we shall prove the following theorem.
 
\begin{teo}\label{teo: ejemplos para Hcal}
	Given $0<\gamma<mn$ there exists pairs of weights $(w,\vec{v})$ satisfying \eqref{eq: clase Hcal(p,gamma,delta) - m} for every $\vec{p}$ and $\delta$ such that $\delta\leq \min\{1,\gamma-n/p\}$, excluding the case $\delta=1$ when $\gamma-n/p=1$.
\end{teo}

The following figure shows the area in which we can find nontrivial weights satisfying condition $\mathcal{H}_m(\vec{p},\gamma,\delta)$, split into the cases $\gamma<1, \gamma=1$ and $\gamma>1$.

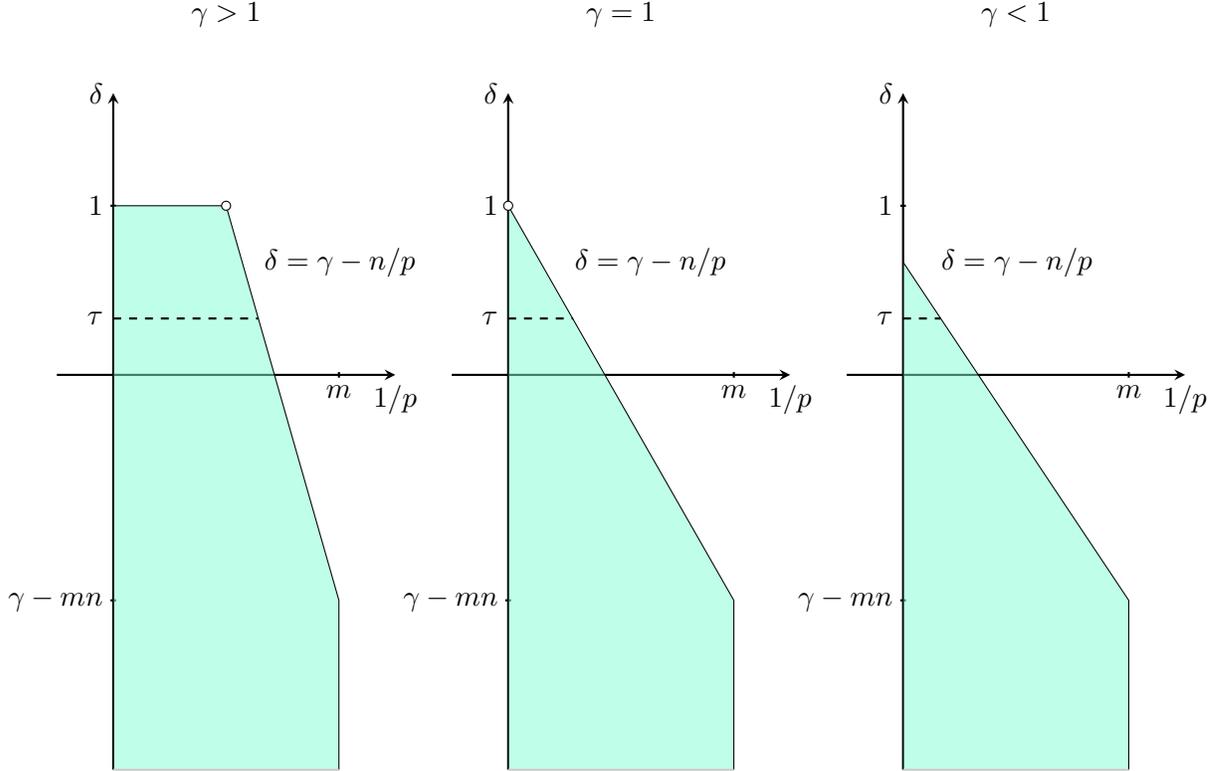
\begin{center}
	\begin{tikzpicture}[scale=0.75]
	% CASO GAMMA>1
	\node[above] at (-4,6) {$\gamma>1$};
	\draw [-stealth, thick] (-6,-7)--(-6,5);
	\draw [-stealth, thick] (-7,0)--(-1,0);
	\draw [thick] (-6.05,3)--(-5.95, 3);
	\node [left] at (-6,5) {$\delta$};
	\node [left] at (-6,3) {$1$};
	\node [below] at (-1,0) {$1/p$};
	\node [below] at (-2,0) {$m$};
	\draw [thick] (-2,0.05)--(-2,-0.05);
	\draw [thick] (-6.05,-4)--(-5.95, -4);
	\node [left] at (-6,-4) {$\gamma-mn$};
	\draw [fill=aquamarine, fill opacity=0.5] (-6,-7)--(-6,3)--(-4,3)--(-2,-4)--(-2,-7)--cycle;
	\draw [color=white] (-2,-7)--(-6,-7);
	% Y=-7/2x-11
	\draw [dashed, thick] (-6,1)--(-3.4286,1);
	\node [left] at (-6,1) {$\tau$};
	\draw [fill=white] (-4,3) circle (0.08cm);
	\node [right] at (-3.5,2) {$\delta=\gamma-n/p$};
	% CASO GAMMA=1
	\node[above] at (3,6) {$\gamma=1$};
	\draw [-stealth, thick] (1,-7)--(1,5);
	\draw [-stealth, thick] (0,0)--(6,0);
	\draw [thick] (0.95,3)--(1.05, 3);
	\node [left] at (1,5) {$\delta$};
	\node [left] at (1,3) {$1$};
	\node [below] at (6,0) {$1/p$};
	\node [below] at (5,0) {$m$};
	\draw [thick] (5,0.05)--(5,-0.05);
	\draw [thick] (0.95,-4)--(1.05, -4);
	\node [left] at (1,-4) {$\gamma-mn$};
	\draw [fill=aquamarine, fill opacity=0.5] (1,-7)--(1,3)--(5,-4)--(5,-7)--cycle;
	\draw [color=white] (5,-7)--(1,-7);
	% Y=-7/4x+19/4
	\draw [dashed, thick] (1,1)--(2.152857,1);
	\node [left] at (1,1) {$\tau$};
	\draw [fill=white] (1,3) circle (0.08cm);
	\node [right] at (2,2) {$\delta=\gamma-n/p$};
	% CASO GAMMA<1
	\node[above] at (10,6) {$\gamma<1$};
	\draw [-stealth, thick] (8,-7)--(8,5);
	\draw [-stealth, thick] (7,0)--(13,0);
	\draw [thick] (7.95,3)--(8.05, 3);
	\node [left] at (8,5) {$\delta$};
	\node [left] at (8,3) {$1$};
	\node [below] at (13,0) {$1/p$};
	\node [below] at (12,0) {$m$};
	\draw [thick] (12,0.05)--(12,-0.05);
	\draw [thick] (7.95,-4)--(8.05, -4);
	\node [left] at (8,-4) {$\gamma-mn$};
	\draw [fill=aquamarine, fill opacity=0.5] (8,-7)--(8,2)--(12,-4)--(12,-7)--cycle;
	\draw [color=white] (12,-7)--(8,-7);
	% Y=-3/2x+14
	\draw [dashed, thick] (8,1)--(8.6667,1);
	\node [left] at (8,1) {$\tau$};
	\node [right] at (8.5,2) {$\delta=\gamma-n/p$};
	\end{tikzpicture}
\end{center}

The following lemma will be useful in order to prove Theorem~\ref{teo: ejemplos para Hcal} (see \cite{Pradolini01}).

\begin{lema}\label{lema: estimacion de la integral de |x|^a en una bola}
	If $R>0$, $B=B(x_B,R)$ is a ball in $\mathbb{R}^n$ and $\alpha>-n$ then
	\[\int_B |x|^{\alpha}\,dx\approx R^n\left(\max\{R,|x_B|\}\right)^\alpha.\]
\end{lema} 

\medskip

\begin{proof}[Proof of Theorem~\ref{teo: ejemplos para Hcal}]
	In \cite{BPR22} we exhibited examples of weights in the class $\mathbb{H}_m(\vec{p},\gamma,\delta)$ given by \eqref{eq: clase Hbb(p,gamma,delta) - m}, for $\gamma-mn\leq \delta\leq \min\{1,\gamma-n/p\}$, excluding the case $\delta=1$ when $\gamma-n/p=1$. By Remark~\ref{obs: Hbb contenida en Hcal} the same examples satisfy
	$\mathcal{H}_m(\vec{p},\gamma,\delta)$, so it will be enough  to check the case $\delta<\gamma-mn$.   
	
 Recall that $\theta_i=n/p_i+(1-\gamma)/m$ and $\mathcal{I}_1=\{1\leq i\leq m: p_i=1\}$. Let us first assume that $\mathcal{I}_1\neq \emptyset$. We choose $-\theta_i<\beta_i<n/p_i'$ for every $i\in\mathcal{I}_2$ and $\theta_i<0$, and $0<\beta_i<n/p_i'$ if $\theta_i\geq 0$. This election implies that 
\[\nu=\sum_{i\in\mathcal{I}_2,\theta_i\geq 0}\beta_i+\sum_{i\in\mathcal{I}_2,\theta_i<0}(\beta_i+\theta_i)>0.\]
We now choose  
\[0<\beta<\min\left\{\frac{\nu}{m_1}, n+\frac{1-\gamma}{m}\right\},\]
and take $\beta_i=-\beta$ for every $i\in\mathcal{I}_1$. Let $\alpha=\delta+\sum_{i=1}^m\beta_i+n/p-\gamma$ and define
\[w(x)=|x|^\alpha\quad\textrm{ and }\quad v_i(x)=|x|^{\beta_i},\quad \textrm{ for } 1\leq i\leq m.\]
 Notice that 
 \[\alpha=\delta+\sum_{i=1}^m\beta_i+n/p-\gamma<\delta+\sum_{i=1}^m\frac{n}{p_i'}+\frac{n}{p}-\gamma=\delta+mn-\gamma<0,\]
 since $\delta<\gamma-mn$, so $w^{-1}$ is a locally integrable function. On the other hand, since $v_i^{-1}\in \mathrm{RH}_\infty$  for  $i\in\mathcal{I}_1$ the same conclusion holds for these weights. For $i\in\mathcal{I}_2$ we also have that $v_i^{-p_i'}$ is locally integrable since
 $\beta_i<n/p_i'$. Therefore, by virtue of Lemma~\ref{lema: equivalencia con local y global}, it will be enough to show that there exists a positive constant $C$ such that the inequality
\begin{equation}\label{eq: teo: ejemplos para Hcal - eq1}
\frac{|B|^{1+(1-\delta)/n}}{w^{-1}(B)}\prod_{i\in\mathcal{I}_1}\left\|\frac{v_i^{-1}\mathcal{X}_{\mathbb{R}^n\backslash B}}{|x_B-\cdot|^{n-\gamma/m+1/m}}\right\|_\infty\,\prod_{i\in\mathcal{I}_2}\left(\int_{\mathbb{R}^n\backslash B}\frac{v_i^{-p_i'}}{|x_B-\cdot|^{(n-\gamma/m+1/m)p_i'}}\right)^{1/p_i'}\leq C
\end{equation}
holds for every ball $B=B(x_B,R)$. We shall first assume that $|x_B|\leq R$. By Lemma~\ref{lema: estimacion de la integral de |x|^a en una bola} we have that
\begin{equation}\label{eq: teo: ejemplos para Hcal - eq2}
\frac{|B|^{1+(1-\delta)/n}}{w^{-1}(B)}\lesssim R^{1-\delta+\alpha}.
\end{equation}
On the other hand, if $i\in\mathcal{I}_1$ and $B_k=B(x_B,2^kR)$, $k\in\mathbb{N}$, we have
\begin{align*}
\left\|\frac{v_i^{-1}\mathcal{X}_{\mathbb{R}^n\backslash B}}{|x_B-\cdot|^{n-\gamma/m+1/m}}\right\|_\infty&\lesssim \sum_{k=0}^\infty \left\|\frac{v_i^{-1}\mathcal{X}_{B_{k+1}\backslash B_k}}{|x_B-\cdot|^{n-\gamma/m+1/m}}\right\|_\infty\\
&\lesssim \sum_{k=0}^\infty \left(2^kR\right)^{-\beta_i-n+\gamma/m-1/m}\\
&\lesssim   R^{-\beta_i-n+\gamma/m-1/m},
\end{align*}
since $-\beta_i-n+\gamma/m-1/m<0$. This yields
\begin{equation}\label{eq: teo: ejemplos para Hcal - eq3}
\prod_{i\in\mathcal{I}_1}\left\|\frac{v_i^{-1}\mathcal{X}_{\mathbb{R}^n\backslash B}}{|x_B-\cdot|^{n-\gamma/m+1/m}}\right\|_\infty\lesssim R^{-\sum_{i\in\mathcal{I}_1}(\beta_i+\theta_i)}.
\end{equation}
Finally, since $\beta_i+\theta_i>0$ for $i\in\mathcal{I}_2$, by Lemma~\ref{lema: estimacion de la integral de |x|^a en una bola} we obtain 
\begin{align*}
	\left(\int_{\mathbb{R}^n\backslash B} \frac{v_i^{-p_i'}(y)}{|x_B-y|^{(n-\gamma/m+1/m)p_i'}}\,dy\right)^{1/p_i'}&\lesssim \sum_{k=0}^\infty (2^kR)^{-n+\gamma/m-1/m}\left(\int_{B_{k+1}\backslash B_k} |y|^{-\beta_ip_i'}\,dy\right)^{1/p_i'}\\
	&\lesssim \sum_{k=0}^\infty (2^kR)^{-n+\gamma/m-1/m-\beta_i+n/p_i'}\\
	&\lesssim R^{-n/p_i+\gamma/m-1/m-\beta_i},
\end{align*}
since $-n/p_i+\gamma/m-1/m-\beta_i<0$ by the choice of $\beta_i$. Therefore, we obtain
\begin{equation}\label{eq: teo: ejemplos para Hcal - eq4}
\prod_{i\in\mathcal{I}_2}\left(\int_{\mathbb{R}^n\backslash B}\frac{v_i^{-p_i'}(y)}{|x_B-y|^{(n-\gamma/m+1/m)p_i'}}\,dy\right)^{1/p_i'}\lesssim R^{-\sum_{i\in\mathcal{I}_2}(\beta_i+\theta_i)}.
\end{equation}
By combining \eqref{eq: teo: ejemplos para Hcal - eq2}, \eqref{eq: teo: ejemplos para Hcal - eq3} and \eqref{eq: teo: ejemplos para Hcal - eq4}, the left-hand side of \eqref{eq: teo: ejemplos para Hcal - eq1} is bounded by
\[CR^{1-\delta+\alpha-\sum_{i=1}^m(\theta_i+\beta_i)}=C.\] 
Now we consider the case $|x_B|>R$. By Lemma~\ref{lema: estimacion de la integral de |x|^a en una bola} we have that
\begin{equation}\label{eq: teo: ejemplos para Hcal - eq5}
\frac{|B|^{1+(1-\delta)/n}}{w^{-1}(B)}\lesssim R^{1-\delta}|x_B|^\alpha\lesssim R^{1-\delta+\alpha},
\end{equation}
because $\alpha<0$. Since $|x_B|>R$, there exists a number $N\in\mathbb{N}$ such that $2^NR<|x_B|\leq 2^{N+1}R$. For $i\in\mathcal{I}_1$ we write
\begin{align*}
\left\|\frac{v_i^{-1}\mathcal{X}_{\mathbb{R}^n\backslash B}}{|x_B-\cdot|^{n-\gamma/m+1/m}}\right\|_\infty&\lesssim \sum_{k=0}^N \left\|\frac{v_i^{-1}\mathcal{X}_{B_{k+1}\backslash B_k}}{|x_B-\cdot|^{n-\gamma/m+1/m}}\right\|_\infty+\sum_{k=N+1}^\infty \left\|\frac{v_i^{-1}\mathcal{X}_{B_{k+1}\backslash B_k}}{|x_B-\cdot|^{n-\gamma/m+1/m}}\right\|_\infty\\
&=S_1^i+S_2^i.
\end{align*}
By standard estimation we have that
\[S_1^i\lesssim |x_B|^{-\beta_i}\sum_{k=0}^N \left(2^kR\right)^{-n+\gamma/m-1/m}\lesssim |x_B|^{-\beta_i}R^{-n+\gamma/m-1/m}=|x_B|^{-\beta_i}R^{-\theta_i}\]
and
\begin{align*}
S_2^i&\lesssim \sum_{k=N+1}^\infty \left(2^kR\right)^{-\beta_i-n+\gamma/m-1/m}\\
&\lesssim \left(2^{N}R\right)^{-\beta_i-n+\gamma/m-1/m}\sum_{k=0}^\infty 2^{k(-\beta_i-n+\gamma/m-1/m)}\\
&\lesssim |x_B|^{-\beta_i}R^{-n+\gamma/m-1/m}
=|x_B|^{-\beta_i}R^{-\theta_i}.
\end{align*}
These inequalities imply that
\begin{equation}\label{eq: teo: ejemplos para Hcal - eq6}
\prod_{i\in\mathcal{I}_1}\left\|\frac{v_i^{-1}\mathcal{X}_{\mathbb{R}^n\backslash B}}{|x_B-\cdot|^{n-\gamma/m+1/m}}\right\|_\infty\lesssim |x_B|^{-\sum_{i\in\mathcal{I}_1}\beta_i}\,\,R^{-\sum_{i\in\mathcal{I}_1}\theta_i}.
\end{equation}

If $i\in\mathcal{I}_2$ we split the integral in a similar way to get
\begin{align*}
	\left(\int_{\mathbb{R}^n\backslash B} \frac{v_i^{-p_i'}(y)}{|x_B-y|^{(n-\gamma/m+1/m)p_i'}}\,dy\right)^{1/p_i'}&\lesssim \sum_{k=0}^\infty(2^{k}R)^{-n+\gamma/m-1/m}\left(\int_{B_k} |y|^{-\beta_ip_i'}\,dy\right)^{1/p_i'}\\
	&=\sum_{k=0}^N+\sum_{k=N+1}^\infty\\
	&=S_1^i+S_2^i.
\end{align*}\label{pag: estimacion del producto para i fuera de I_1, |x_B|>R}
 We shall estimate the sum $S_1^i+S_2^i$ by distinguishing into the cases $\theta_i<0$, $\theta_i=0$ and $\theta_i>0$. Let us first assume that $\theta_i<0$. Then by Lemma~\ref{lema: estimacion de la integral de |x|^a en una bola} we obtain
\begin{align*}
	S_1^i&\lesssim \sum_{k=0}^N(2^{k}R)^{-n+\gamma/m-1/m+n/p_i'}|x_B|^{-\beta_i}\\
	&\lesssim |x_B|^{-\beta_i}R^{-\theta_i}\sum_{k=0}^N 2^{-k\theta_i}\\
	&\lesssim |x_B|^{-\beta_i}(2^NR)^{-\theta_i}\\
	&\lesssim |x_B|^{-\beta_i-\theta_i},
\end{align*}
since $\theta_i<0$. For $S_2^i$ we apply again Lemma~\ref{lema: estimacion de la integral de |x|^a en una bola} in order to get
\begin{align*}
	S_2^i&\lesssim \sum_{k=N+1}^\infty(2^{k}R)^{-n+\gamma/m-1/m+n/p_i'-\beta_i}\\
	&\lesssim \sum_{k=N+1}^\infty \left(2^{k}R\right)^{-\beta_i-\theta_i}\\
	&= \left(2^{N+1}R\right)^{-\beta_i-\theta_i}\sum_{k=0}^\infty 2^{-k(\beta_i+\theta_i)}\\
	&\lesssim |x_B|^{-\beta_i-\theta_i},
\end{align*}
since $\theta_i+\beta_i>0$. This yields
\begin{equation}\label{eq: teo: ejemplos para Hcal - eq7}
	S_1^i+S_2^i\lesssim |x_B|^{-\beta_i-\theta_i}
\end{equation}
when $\theta_i<0$.

Now assume that $\theta_i=0$. By proceeding similarly as in the previous case, we have
\[S_1^i\lesssim |x_B|^{-\beta_i}N\lesssim |x_B|^{-\beta_i}\log_2\left(\frac{|x_B|}{R}\right),\]
and
\[S_2^i\lesssim |x_B|^{-\beta_i}\]
since $\beta_i>0$ when $\theta_i=0$. Consequently,
\begin{equation}\label{eq: teo: ejemplos para Hcal - eq8}
	S_1^i+S_2^i\lesssim |x_B|^{-\beta_i}\left(1+\log_2\left(\frac{|x_B|}{R}\right)\right)\lesssim |x_B|^{-\beta_i}\log_2\left(\frac{|x_B|}{R}\right).
\end{equation}

We finally consider the case $\theta_i>0$. For $S_2^i$ we can proceed exactly as in the case $\theta_i<0$ and get the same bound. On the other hand, for $S_1^i$ we have that
\begin{align*}
	S_1^i&\lesssim \sum_{k=0}^N(2^{k}R)^{-n+\gamma/m-1/m+n/p_i'}|x_B|^{-\beta_i}\\
	&\lesssim |x_B|^{-\beta_i}R^{-\theta_i}\sum_{k=0}^N 2^{-k\theta_i}\\
	&\lesssim |x_B|^{-\beta_i}\left(2^NR\right)^{-\theta_i}2^{N\theta_i}\\
	&\lesssim |x_B|^{-\beta_i-\theta_i}2^{N\theta_i}.
\end{align*}\label{pag: estimacion de S_1^i y S_2^i,  theta_i>0}
Therefore, if $i\in\mathcal{I}_2$ and $\theta_i>0$ we get
\begin{equation}\label{eq: teo: ejemplos para Hcal - eq9}
	S_1^i+S_2^i\lesssim |x_B|^{-\beta_i-\theta_i}\left(1+2^{N\theta_i}\right)\lesssim 2^{N\theta_i}|x_B|^{-\beta_i-\theta_i}. 
\end{equation}
By combining \eqref{eq: teo: ejemplos para Hcal - eq7},\eqref{eq: teo: ejemplos para Hcal - eq8} and \eqref{eq: teo: ejemplos para Hcal - eq9}  we obtain
\begin{align*}
	\prod_{i\in\mathcal{I}_2}\left(\int_{\mathbb{R}^n\backslash B} \frac{v_i^{-p_i'}(y)}{|x_B-y|^{(n-\gamma/m+1/m)p_i'}}\,dy\right)^{1/p_i'}&\lesssim \prod_{i\in\mathcal{I}_2, \theta_i<0} |x_B|^{-\beta_i-\theta_i} \prod_{i\in\mathcal{I}_2, \theta_i=0} |x_B|^{-\beta_i}\log_2\left(\frac{|x_B|}{R}\right) \\
	&\qquad\times\prod_{i\in\mathcal{I}_2, \theta_i>0} |x_B|^{-\beta_i-\theta_i}2^{N\theta_i} \\
	&\lesssim |x_B|^{-\sum_{i\in\mathcal{I}_2}(\beta_i+\theta_i)}2^{N\sum_{i\in\mathcal{I}_2,\theta_i> 0}\theta_i}\\
	&\qquad \times\left(\log_2\left(\frac{|x_B|}{R}\right)\right)^{\#\{i\in\mathcal{I}_2, \theta_i=0\}}.
\end{align*}

The estimate above combined with \eqref{eq: teo: ejemplos para Hcal - eq5} and \eqref{eq: teo: ejemplos para Hcal - eq6} allows us to bound the left-hand side of \eqref{eq: teo: ejemplos para Hcal - eq1} by
\[CR^{1-\delta+\alpha} |x_B|^{-\sum_{i\in\mathcal{I}_1}\beta_i}R^{-\sum_{i\in\mathcal{I}_1}\theta_i}|x_B|^{-\sum_{i\in\mathcal{I}_2}(\beta_i+\theta_i)}2^{N\sum_{i\in\mathcal{I}_2,\theta_i> 0}\theta_i}
\left(\log_2\left(\frac{|x_B|}{R}\right)\right)^{\#\{i\in\mathcal{I}_2, \theta_i=0\}}\]
or equivalently by
\begin{equation}\label{eq: teo: ejemplos para Hcal - eq10}
\left(\frac{R}{|x_B|}\right)^{1-\delta+\alpha-\sum_{i\in\mathcal{I}_1}\theta_i-\sum_{i\in\mathcal{I}_2,\theta_i> 0}\theta_i}\left(\log_2\left(\frac{|x_B|}{R}\right)\right)^{\#\{i\in\mathcal{I}_2, \theta_i=0\}}.	
\end{equation}
Notice that the exponent of $R/|x_B|$ is equal to 
\[\sum_{i\in\mathcal{I}_2,\theta_i<0}(\beta_i+\theta_i)+\sum_{i\in\mathcal{I}_1}\beta_i+\sum_{i\in\mathcal{I}_2,\theta_i\geq 0}\beta_i=\nu-m_1\beta>0,\]
from our election of $\beta$. Since $\log t\lesssim \varepsilon^{-1}t^\varepsilon$ for every $t\geq 1$ and every $\varepsilon>0$, we can bound \eqref{eq: teo: ejemplos para Hcal - eq10} by
\[C\left(\frac{R}{|x_B|}\right)^{\nu-m_1\beta-\varepsilon\#\{i\in\mathcal{I}_2, \theta_i=0\}},\]
and this exponent is positive provided we choose $\varepsilon>0$ sufficiently small. The proof is complete when $\mathcal{I}_1\neq\emptyset$. Otherwise, we can follow the same steps and define the same parameters, omitting the factor corresponding to $\mathcal{I}_1$. This concludes the proof. 
\end{proof}

We finish with the proof of the theorem dealing with the case $w=\prod_{i=1}^m v_i$.

\begin{proof}[Proof of Theorem~\ref{teo: caso de pesos iguales}]
Let $\alpha=p/(mp-1)$ and assume that $\alpha>1$. If $\vec{v}\in\mathcal{H}_m(\vec{p},\gamma,\delta)$, then by condition \eqref{eq: condicion local 2} we get 
\begin{equation}\label{eq: teo: caso de pesos iguales - eq1}
|B|^{-\delta/n+\gamma/n-1/p}\prod_{i\in\mathcal{I}_1}\|v_i^{-1}\mathcal{X}_B\|_\infty\,\prod_{i\in\mathcal{I}_2}\left(\frac{1}{|B|}\int_B v_i^{-p_i'}\right)^{1/p_i'}\leq \frac{C}{|B|}\int_B \prod_{i=1}^m v_i^{-1}.
\end{equation}
Notice that $\sum_{i=1}^m \alpha/p_i'=1$. Therefore we apply Hölder inequality with $p_i'/\alpha$ in order to obtain 
\[\left(\frac{1}{|B|}\int_B \left(\prod_{i=1}^m v_i^{-1}\right)^\alpha\right)^{1/\alpha}\leq \prod_{i\in\mathcal{I}_1}\|v_i^{-1}\mathcal{X}_B\|_\infty\,\prod_{i\in\mathcal{I}_2}\left(\frac{1}{|B|}\int_B v_i^{-p_i'}\right)^{1/p_i'}.\]
By multiplying each side of the inequality above by $|B|^{-\delta/n+\gamma/n-1/p}$ and using \eqref{eq: teo: caso de pesos iguales - eq1} we get
\[|B|^{-\delta/n+\gamma/n-1/p}\left(\frac{1}{|B|}\int_B \left(\prod_{i=1}^m v_i^{-1}\right)^\alpha\right)^{1/\alpha}\leq \frac{C}{|B|}\int_B \prod_{i=1}^m v_i^{-1}.\]
From this estimate we can conclude that
\[|B|^{-\delta/n+\gamma/n-1/p}\leq C\]
for every ball $B$, since $\alpha>1$. Then we must have that $\delta/n=\gamma/n-1/p$.	
\end{proof}

%BIBLIOGRAFIA
		
\def\cprime{$'$}
\providecommand{\bysame}{\leavevmode\hbox to3em{\hrulefill}\thinspace}
\providecommand{\MR}{\relax\ifhmode\unskip\space\fi MR }
\providecommand{\MRhref}[2]{%
	\href{http://www.ams.org/mathscinet-getitem?mr=#1}{#2}
}
\providecommand{\href}[2]{#2}

\end{document}